\documentclass{sigma}
\usepackage{graphics,epsfig,color}
\usepackage{hyperref}

\newcommand{\R}{{\mathbb R}}

\newcommand{\C}{{\mathbb C}}

\def\cprime{$'$}

\makeatletter
\def\eqnarray{\stepcounter{equation}\let\@currentlabel=\theequation
\global\@eqnswtrue
\tabskip\@centering\let\\=\@eqncr
$$\halign to \displaywidth\bgroup\hfil\global\@eqcnt\z@
  $\displaystyle\tabskip\z@{##}$&\global\@eqcnt\@ne
  \hfil$\displaystyle{{}##{}}$\hfil
  &\global\@eqcnt\tw@ $\displaystyle{##}$\hfil
  \tabskip\@centering&\llap{##}\tabskip\z@\cr}

\def\endeqnarray{\@@eqncr\egroup
      \global\advance\c@equation\m@ne$$\global\@ignoretrue}

\def\@yeqncr{\@ifnextchar [{\@xeqncr}{\@xeqncr[5pt]}}
\makeatother

\begin{document}

\renewcommand{\PaperNumber}{***}

\FirstPageHeading

\ShortArticleName{Convergence of Magnus integral addition theorems}

\ArticleName{Convergence of Magnus integral addition theorems 
for confluent hypergeometric functions}

\Author{Howard S.~COHL, $^\dag\!\!\ $ Jessie E.~HIRTENSTEIN, 
$^\ast\!\!\ $ and Hans VOLKMER $^\ddag$}

\AuthorNameForHeading{H.~S.~Cohl \& H. Volkmer}

\Address{$^\dag$~Applied and Computational Mathematics Division,
National Institute of Standards and Technology,
Gaithersburg, MD, 20899-8910, U.S.A.
} 
\EmailD{howard.cohl@nist.gov} 
\URLaddressD{http://www.nist.gov/itl/math/msg/howard-s-cohl.cfm}

\Address{$^\ast$~Department of Physics,
American University, Washington DC, 20016, U.S.A.
}

\Address{$^\ddag$~Department of Mathematical Sciences, 
University of Wisconsin--Milwaukee, P.O.~Box 413, Milwaukee, WI, 53201, U.S.A.
}
\EmailD{volkmer@uwm.edu} 


\ArticleDates{Received XX July 2015 in final form ????; Published online ????}

\Abstract{
In 1946, Magnus presented an addition theorem
for the confluent hypergeometric function of the second kind $U$ with
argument $x+y$ expressed as an integral of a product of two $U$'s, one with
argument $x$ and another with argument $y$.
We take advantage of recently obtained asymptotics for $U$ with
large complex first parameter to determine a domain of convergence
for Magnus' result.
Using well-known specializations of $U$, we obtain corresponding integral addition
theorems with precise domains of convergence for modified parabolic cylinder
functions, and Hankel, Macdonald, and Bessel functions of the first and second kind
with order zero and one.
}


\Classification{33C10; 33C15}


\section{Introduction}
In 1941 Magnus \cite[(11)]{Magnus41} derived an integral addition theorem for the order zero Hankel function of the second kind.
In 1946 \cite{Magnus}, he generalized his addition theorem to Whittaker $W$-functions. In terms of the confluent hypergeometric function 
of the second kind $U:\C\times\C\times(\C\setminus(-\infty,0])\to\C$
\cite[(13.2.6)]{NIST} \cite{NIST:DLMF},
this result takes the form
\begin{equation}\label{1:add}
\Gamma(c)U(c,2c,x+y)=\frac{1}{4\pi}\int_{-\infty}^\infty \Gamma\left(\tfrac{c-it}{2}\right) \Gamma\left(\tfrac{c+it}{2}\right)
U\left(\tfrac{c-it}{2},c,x\right)U\left(\tfrac{c+it}{2},c,y\right)dt ,
\end{equation}
where $\Re c>0$, $x,y>0$, and 
$\Gamma:{\mathbb C}\setminus\{0,-1,-2,\ldots\}\to\C$ denotes 
Euler's gamma function.
Magnus' earlier addition theorem \cite{Magnus41} is obtained from \eqref{1:add} by setting $c=\frac12$. In this special case
the confluent hypergeometric functions appearing under the integral in \eqref{1:add} reduce to modified parabolic cylinder functions \cite[(27--28)]{CohlVolkmer}.
Following \cite{CohlVolkmer}, we refer to modified parabolic cylinder functions
as those with arguments on complex straight lines $\frac{\pi}{4}$ radians off the real
and imaginary axes.
In \cite[p.~355]{Magnus41}, it is stated that in the special case $c=\frac12$, \eqref{1:add} remains valid for certain complex values of
$x,y$. Buchholz \cite[(10b), p.~163]{Buchholz} also proved \eqref{1:add} using the notation of Whittaker functions with the constraints $x,y>0$.
It is also given in Erd{\'e}lyi {\it et al.}~(1981) \cite[(6.15.2.15)]{Erdelyi}, but there
the equation is written incorrectly.

The goal of this paper is to determine the maximal domain for $(x, y)$ on which \eqref{1:add} is valid.
In addition, we also discuss the special cases for order zero
and one Hankel functions of the first and
second kind, Bessel functions of
the first and second kind
and modified Bessel functions of the second kind
\cite[Section 10.2(ii), (10.25.3)]{NIST}.

In order to find the maximal domain on which \eqref{1:add} is valid, we need to find the values of $(x,y)$ for which the integral in \eqref{1:add} converges.
This requires knowledge of the asymptotic behavior of the function $\Gamma(a) U(a,b,z)$ when $|a|\to\infty$ while $b,z$ are fixed.

\section{Confluent Hypergeometric Functions}
The confluent hypergeometric function 
of the second kind $U$ is defined and analytic for $a,b\in\C$ and $z$ on the Riemann surface of the logarithm.
If $\Re a>0$, $|\arg z|<\frac\pi 2$, then it is given by a Laplace transform \cite[(13.4.4)]{NIST}
\begin{equation}\label{2:U}
\Gamma(a) U(a,b,z)=\int_0^\infty e^{-zt}t^{a-1}(1+t)^{b-a-1}\,dt .
\end{equation}

The asymptotic behavior of $U$ for large values of $a$ can be expressed in terms of the modified Bessel function of the second kind
$K_\nu$. This function can be defined in terms of $U$ as
\cite[(6.9.1.13)]{Erdelyi}
\begin{equation}\label{2:K}
K_{\nu}(z):=\sqrt{\pi}e^{-z}\,(2z)^{\nu}\,U\left(\tfrac12+\nu,1+2\nu,2z\right).
\end{equation}
If $b\in\C$ and $z>0$, it is known \cite{Temme} that
\begin{eqnarray}
&&2^{b-2}u^{1-b}e^{-\frac12z^2}z^b\Gamma(a)U(a,b,z^2)=\nonumber\\
&&\hspace{1cm} zK_{b-1}(uz)\left(1+\mathcal{O}\left(\frac1{u^2}\right)\right)-\frac{z}{u}K_b(uz)\left(\frac16 z^3+\mathcal{O}\left(\frac1{u^2}\right)\right),\label{2:Uasy1}
\end{eqnarray}
where $a=\frac14 u^2+\frac12b$ and $0<u\to\infty$.
However, we need \eqref{2:Uasy1} for $u\to\infty$ along the rays $\arg u=\pm \frac\pi4$, and we have to allow complex $z$.
In \cite{Volkmer2016} (see also \cite{Olver56})
the following result is proved.

\begin{lemma}\label{2:l1}
Let $b\in\C$, $\theta\in(-\frac\pi2,\frac\pi2)$, $R>0$, and set $a=\frac14u^2+\frac12b$, $u=|u|e^{i\theta}$. Then, as $|u|\to\infty$, \eqref{2:Uasy1} holds uniformly for $0<|z|\le R$, $-\infty<\arg z<\infty$.
\end{lemma}

When we combine Lemma \ref{2:l1} with the well-known asymptotic expansion
\cite[(10.40.2)]{NIST}
\[ K_\nu(w)=\left(\frac{\pi}{2w}\right)^{1/2}e^{-w}
\left(1+\mathcal{O}\left(\frac1w\right)\right)
\quad \text{as $|w|\to\infty,\ |\arg w|\le \tfrac32\pi-\delta$,\ $\delta>0$} ,\]
we obtain the following lemma.

\begin{lemma}\label{2:Uasylemma}
Let $b\in\C$, $\theta\in(-\frac\pi2,\frac\pi2)$, $0<r<R$, and set $a=\frac14u^2+\frac12b$, $u=|u|e^{i\theta}$. Then, as $|u|\to\infty$,
\begin{equation}\label{2:Uasy}
\Gamma(a)U(a,b,z^2)=\sqrt{\pi}\,\left(\tfrac{u}{2}\right)^{b-\frac32}\,z^{\frac12-b}\,e^{\frac12 z^2-uz}\left(1+\mathcal{O}\left(\frac1u \right)\right)
\end{equation}
holds uniformly for $r\le |z|\le R$ and $|\arg z|\le \pi$.
\end{lemma}

\begin{theorem}\label{2:addthm}
Let $c\in\C$ with $\Re c>0$. Suppose that $x,y\in \C\setminus\{0\}$, $|\arg x|\le 2\pi$, $|\arg y|\le 2\pi$ satisfy
\begin{equation}\label{2:cond}
 \Re\left((1-i)\sqrt{x}+(1+i)\sqrt y\right)>0\text{ and }  \Re\left((1+i)\sqrt{x}+(1-i)\sqrt y\right)>0 .
\end{equation}
Let $\sigma\in\C$ be such that $|\Re\sigma|<\frac12 \Re c$.
Then,
\begin{equation}
\label{2:add}
\Gamma(c)U(c,2c,x+y)=\frac{1}{2\pi i}
\int_{\sigma-i\infty}^{\sigma+i \infty}
\Gamma\left(\tfrac{c}{2}-s\right)\Gamma\left(\tfrac{c}{2}+s\right)
U\left(\tfrac{c}{2}-s,c,x\right) U\left(\tfrac{c}{2}+s,c,y\right)
ds.
\end{equation}
Remark: Since
\[ \left((1-i)\sqrt{x}+(1+i)\sqrt y\right)\left((1+i)\sqrt{x}+(1-i)\sqrt y\right)=2(x+y),\]
\eqref{2:cond} implies that $x+y$ lies in the cut plane $\C\setminus(-\infty,0]$. The function $U(a,b,z)$ on the left-hand side of \eqref{2:add} is evaluated in the sector
$-\pi<\arg z<\pi$ (principal values.)
However, the functions $U$ under the integral sign may attain non-principal values.
\end{theorem}
\begin{proof}
We consider first the case $\sigma=0$, and set $s=\frac12 it$. When $\pm t>0$ we use Lemma \ref{2:Uasylemma} with $b=c$ and
$u=\sqrt{2|t|} e^{\pm \frac14\pi i}$. Then we obtain, as $t\to\pm\infty$,
\begin{equation}
\Gamma\left(\tfrac{c+it}{2}\right) U\left(\tfrac{c+it}{2},c,z\right)=\sqrt{\pi}\,
\left(\tfrac{|t|}{2}e^{\pm \frac12 i\pi}\right)^{\frac12c-\frac34} \,z^{\frac14-\frac12c}\,e^{\frac12z-(1\pm i)\sqrt{|t|z}}\left(1+\mathcal{O}(|t|^{-\frac12})\right)
\end{equation}
which when applied to the integrand of \eqref{2:add} gives
\begin{multline}\label{3:asyintegrand}
\frac{1}{4\pi}
\Gamma\left(\tfrac{c-it}{2}\right)\Gamma\left(\tfrac{c+it}{2}\right)
U\left(\tfrac{c-it}{2},c,x\right) U\left(\tfrac{c+it}{2},c,y\right)
 =\\
e^{\frac12(x+y)}(xy)^{\frac14-\frac12 c}2^{-\frac12-c}|t|^{c-\frac32}e^{-\sqrt{|t|}\left[(1\mp i)\sqrt{x}+(1\pm i)\sqrt{y}\right]}
\left(1+\mathcal{O}(|t|^{-\frac12})\right).
\end{multline}
This shows that the integral on the right-hand side of \eqref{2:add} converges provided that $x,y$ satisfy \eqref{2:cond}.
Since \eqref{3:asyintegrand} holds locally uniformly, by Weierstrass' theorem, the right-hand side of \eqref{2:add} (with $\sigma=0$)  is an analytic function of $(x,y)$ on the connected domain $D$ consisting of all $(x,y)$ satisfying \eqref{2:cond}.
By the known result \eqref{1:add}, \eqref{2:add} is true for $\sigma=0$ and $x,y>0$. Hence, by the identity theorem for analytic functions, \eqref{2:add} with $\sigma=0$ holds for all $(x,y)\in D$.

If $\Re\sigma\ne 0$ we apply the Cauchy integral theorem to the rectangle $|\Re s|\le |\Re\sigma|$, $|\Im s|\le\tau$ and let $\tau\to\pm \infty$.
To justify the procedure we establish the estimate
\begin{equation}\label{2:est}
 |\Gamma(a)U(a,c,z^2)u^{\frac32-c}e^{uz}|\le L,
\end{equation}
where $a=\frac12 c+\frac14 u^2$, $|u|\ge u_0>0$, $|\arg u|\le \frac12\pi-\delta$, $\delta>0$, and $c$, $z\ne0 $ are fixed. The constant $L$ is independent of $u$.
We know \eqref{2:est} on the boundary of the sectors
$|u|\ge u_0$, $0\le \arg u\le \frac12\pi-\delta$ and $|u|\ge u_0$, $-\frac12\pi+\delta\le \arg u\le 0$  from Lemma \ref{2:Uasylemma}. Then 
we extend it to the full sectors by using the Phragmen-Lindel\"of theorem.
In order to apply this theorem we need the rough estimate (whose proof we omit)
\begin{equation}\label{2:est2}
 |\Gamma(a)U(a,c,z)|\le C_1e^{C_2|a|}
\end{equation}
for $|a|\ge a_0>0$, $|\arg a|\le \pi-\delta$, $\delta>0$, and $c$, $z\ne 0$ fixed.
This completes the proof.
\end{proof}

The parabolic cylinder function $D_{\nu}(z)$ defined in \cite[(6.9.2.31)]{Erdelyi} as
\begin{equation}\label{2:D1}
D_\nu(z):=2^{\nu/2}e^{-z^2/4}\,U
\left(-\tfrac{\nu}{2},\tfrac12,\tfrac12 z^2\right)
\end{equation}
is analytic for $\nu,z\in\C$.
By setting $c=\frac12$ in Theorem \ref{2:addthm} we obtain  the following theorem.

\begin{corollary}\label{2:cor1}
Suppose that $x,y\in \C$, $|\arg x|\le 2\pi$, $|\arg y|\le 2\pi$ satisfy \eqref{2:cond}.
Let $\sigma\in\C$ be such that $|\Re \sigma|<\frac14$.
Then,
\begin{multline}\label{2:add2}
\hspace{0.8cm}
U\left(\tfrac12,1,x+y\right)
=\frac{e^{(x+y)/2}}{2^{1/2}\pi^{3/2} i}
\int_{\sigma-i\infty}^{\sigma+i\infty}
\Gamma\left(\tfrac14-s\right)\Gamma\left(\tfrac14+s\right)\\
\times
D_{2s-\frac12}(\sqrt{2x}) D_{-2s-\frac12}(\sqrt{2y})\,ds .
\end{multline}
\end{corollary}

If $c=\frac32$ we may use
\begin{equation}\label{2:D2}
D_\nu(z)=2^{(\nu-1)/2}e^{-z^2/4}\,z\,
U\left(\tfrac{1-\nu}{2},\tfrac32,\tfrac12 z^2\right),
\end{equation}
which follows from substituting \eqref{2:D1} into \cite[(13.2.40)]{NIST}.

\begin{corollary}\label{2:cor2}
Suppose that $x,y\in \C$, $|\arg x|\le 2\pi$, $|\arg y|\le2\pi$ satisfy \eqref{2:cond}.
Let $\sigma\in\C$ be such that $|\Re \sigma|<\frac34$.
Then,
\begin{multline}
\label{hypergeometric32}
\hspace{0.8cm}
\sqrt{xy}U\left(\tfrac32,3,x+y\right)=
\frac{2^{1/2} e^{(x+y)/2}}{\pi^{3/2} i}
\int_{\sigma-i\infty}^{\sigma+i \infty}
\Gamma\left(\tfrac{3}{4}-s\right)\Gamma\left(\tfrac{3}{4}+s\right)\\
\times
D_{2s-\frac12}(\sqrt{2x})D_{-2s-\frac12}(\sqrt{2y})
ds.
\end{multline}
\end{corollary}

\section{Bessel Functions}
We next derive integral addition theorems for Hankel functions which are
defined by \cite[(6.9.1.12)]{Erdelyi}, namely
\begin{gather}
\label{3:Kummer2Hankel1}
H_{\nu}^{(1)}(z):=\frac{-2i}{\sqrt{\pi}}e^{i(z-\nu\pi)}\,(2z)^{\nu}\,
U\left(\tfrac12+\nu,1+2\nu,-2iz\right),\\
\label{3:Kummer2Hankel2}
H_{\nu}^{(2)}(z):=\frac{2i}{\sqrt{\pi}}e^{-i(z-\nu\pi)}\,(2z)^{\nu}\,
U\left(\tfrac12+\nu,1+2\nu,2iz\right).
\end{gather}
In Magnus (1941) \cite[(11)]{Magnus41}, an integral addition theorem for the Hankel
function of the second kind $H_0^{(2)}$ is given as follows, where we have
extended the domain of convergence.
\begin{corollary}\label{3:HankelH02}
Suppose that $\xi,\eta\in\mathbb{C}$ satisfy
\begin{equation}\label{3:cond2}
\Re (\eta+i\xi)>0\text{ and } \Re(\xi+i\eta)>0.
\end{equation}
Let $\sigma\in\C$ be such that $|\Re\sigma|<\frac14$. Then,
\begin{multline}\label{3:addH02}
\hspace{0.8cm}
H_0^{(2)}\left(\tfrac12(\xi^2+\eta^2)\right)=
\frac{2^{1/2}}{\pi^2}
\int_{\sigma-i\infty}^{\sigma+i\infty}
\Gamma\left(\tfrac14-s\right)\Gamma\left(\tfrac14+s\right)\\
\times
D_{2s-\frac12}\left((1+i)\xi\right)D_{-2s-\frac12}\left((1+i)\eta\right)
ds.
\end{multline}
Remark: Since
\[ i(\xi^2+\eta^2)=(\xi+i\eta)(\eta+i\xi), \]
\eqref{3:cond2} implies that $i(\xi^2+\eta^2)\in \C\setminus(-\infty,0]$. The Hankel function $H_0^{(2)}(z)$ on the left-hand side of \eqref{3:addH02} is evaluated
in the sector $-\frac32 \pi<\arg z<\frac12\pi$.
\end{corollary}
\begin{proof}
We use \eqref{3:Kummer2Hankel2} with $\nu=0$ and Corollary \ref{2:cor1} with $x=i\xi^2$ and $y=i\eta^2$. We choose $\arg \xi,\arg \eta\in(-\frac54\pi,\frac34\pi]$.
Then $\arg x, \arg y\in(-2\pi,2\pi]$ and $\sqrt{2x}=(1+i)\xi$, $\sqrt{2y}=(1+i)\eta$.
\end{proof}

In a similar way we obtain an integral addition theorem for the Hankel
function of the first kind.

\begin{corollary}\label{3:HankelH01}
Suppose that $\xi,\eta\in\mathbb{C}$ satisfy
\begin{equation}\label{3:cond1}
\Re (\eta-i\xi)>0\text{ and } \Re(\xi-i\eta)>0.
\end{equation}
Let $\sigma\in\C$ be such that $|\Re\sigma|<\frac14$. Then,
\begin{multline}\label{3:addH01}
\hspace{0.8cm}
H_0^{(1)}\left(\tfrac12(\xi^2+\eta^2)\right)=
-\frac{2^{1/2}}{\pi^2}
\int_{\sigma-i\infty}^{\sigma+i\infty}
\Gamma\left(\tfrac14-s\right)\Gamma\left(\tfrac14+s\right)\\
\times
D_{2s-\frac12}\left((1-i)\xi\right)D_{-2s-\frac12}\left((1-i)\eta\right)
ds.
\end{multline}
Remark: Since
\[ -i(\xi^2+\eta^2)=(\xi-i\eta)(\eta-i\xi), \]
\eqref{3:cond1} implies that $-i(\xi^2+\eta^2)\in \C\setminus(-\infty,0]$. The Hankel function $H_0^{(1)}(z)$ on the left-hand side of \eqref{3:addH01} is evaluated
in the sector $-\frac12 \pi<\arg z<\frac32\pi$.
\end{corollary}

Using \eqref{3:Kummer2Hankel1} and \eqref{3:Kummer2Hankel2} with $\nu=1$, we derive respective integral addition theorems from Corollary \ref{2:cor2} as follows.

\begin{corollary}\label{3:HankelH11}
Suppose that $\xi,\eta\in\mathbb{C}$ satisfy \eqref{3:cond1}.
Let $\sigma\in\C$ be such that $|\Re\sigma|<\frac34$. Then,
\begin{multline}\label{3:addH11}
\hspace{0.8cm}
\xi\eta\, H_1^{(1)}\left(\tfrac12(\xi^2+\eta^2)\right)=
-\frac{2^{3/2}(\xi^2+\eta^2)}{\pi^2 i }
\int_{\sigma-i\infty}^{\sigma+i\infty}
\Gamma\left(\tfrac34-s\right)\Gamma\left(\tfrac34+s\right)\\
\times
D_{2s-\frac12}\left((1-i)\xi\right)D_{-2s-\frac12}\left((1-i)\eta\right)
ds.
\end{multline}
\end{corollary}

\begin{corollary}\label{3:HankelH12}
Suppose that $\xi,\eta\in\mathbb{C}$ satisfy \eqref{3:cond2}.
Let $\sigma\in\C$ be such that $|\Re\sigma|<\frac34$. Then,
\begin{multline}\label{3:addH12}
\hspace{0.8cm}
\xi\eta\, H_1^{(2)}\left(\tfrac12(\xi^2+\eta^2)\right)=
-\frac{2^{3/2}(\xi^2+\eta^2)}{\pi^2 i}
\int_{\sigma-i\infty}^{\sigma+i\infty}
\Gamma\left(\tfrac34-s\right)\Gamma\left(\tfrac34+s\right)\\
\times
D_{2s-\frac12}\left((1+i)\xi\right) D_{-2s-\frac12}\left((1+i)\eta\right)
ds.
\end{multline}
\end{corollary}

We can now derive integral addition theorems for Bessel functions of the first and second
kind $J_{\nu}$ and $Y_{\nu}$, respectively by \cite[(10.4.4)]{NIST}
\begin{gather}
\label{3:HankelBesselJ}
J_{\nu}(z):=\frac12\Big(H_{\nu}^{(1)}(z)+H_{\nu}^{(2)}(z)\Big),\\
\label{3:HankelBesselY}
Y_{\nu}(z):=-\frac{i}{2}\Big(H_{\nu}^{(1)}(z)-H_{\nu}^{(2)}(z)\Big).
\end{gather}

\begin{corollary}\label{3:BesselJ0}
Suppose that $\xi,\eta\in\C$ satisfy $|\Im\eta|<\Re \xi$ and $|\Im \xi|<\Re \eta$.
Let $\sigma\in\C$ be such that $|\Re \sigma|<\frac14$.
Then,
\begin{multline}\label{3:addJ0}
\hspace{0.5cm}
J_0\left(\tfrac12(\xi^2+\eta^2)\right)
=\frac{1}{{2}^{1/2}\pi^2}
\int_{\sigma-i\infty}^{\sigma+i\infty}
\Gamma\left(\tfrac14-s\right)
\Gamma\left(\tfrac14+s\right)\\
\times
\left[
D_{2s-\frac12}((1+i)\xi)
D_{-2s-\frac12}((1+i)\eta)
-
D_{2s-\frac12}((1-i)\xi)
D_{-2s-\frac12}((1-i)\eta)
\right]
ds.
\end{multline}
Remark: The assumptions on $\xi, \eta$ imply that $\Re(\xi^2+\eta^2)>0$. The Bessel function $J_0(z)$ on the left-hand side of \eqref{3:addJ0} is evaluated in the sector $-\frac\pi2<\arg z<\frac\pi2$.
\end{corollary}

\begin{proof}
This follows by substituting \eqref{3:addH01} and \eqref{3:addH02} into \eqref{3:HankelBesselJ} with $\nu=0$.
The stated conditions on $\xi,\eta$ are equivalent to \eqref{3:cond2}, \eqref{3:cond1}.
\end{proof}

Similarly, we obtain integral addition theorems for the Bessel functions $Y_0$, $J_1$ and $Y_1$.
\begin{corollary}\label{3:BesselY0}
Suppose that $\xi,\eta\in\C$ satisfy $|\Im\eta|<\Re \xi$ and $|\Im \xi|<\Re \eta$.
Let $\sigma\in\C$ be such that $|\Re \sigma|<\frac14$. Then,
\begin{multline*}
\hspace{0.5cm}
Y_0\left(\tfrac12(\xi^2+\eta^2)\right)
=\frac{i}{{2}^{1/2}\pi^2}
\int_{\sigma-i\infty}^{\sigma+i\infty}
\Gamma\left(\tfrac14-s\right)
\Gamma\left(\tfrac14+s\right)\\
\times
\Bigl[
D_{2s-\frac12}((1-i)\xi)
D_{-2s-\frac12}((1-i)\eta)
+
D_{2s-\frac12}((1+i)\xi)
D_{-2s-\frac12}((1+i)\eta)
\Bigr]
ds.\nonumber
\end{multline*}
\end{corollary}

\begin{corollary}\label{3:BesselJ1}
Suppose that $\xi,\eta\in\C$ satisfy $|\Im\eta|<\Re \xi$ and $|\Im \xi|<\Re \eta$.
Let $\sigma\in\C$ be such that $|\Re \sigma|<\frac34$.
Then,
\begin{multline*}
\hspace{0.5cm}
J_1\left(\tfrac12(\xi^2+\eta^2)\right)
=\frac{i 2^{1/2}(\xi^2+\eta^2)}{\pi^2\xi\eta}
\int_{\sigma-i\infty}^{\sigma+i\infty}
\Gamma\left(\tfrac34-s\right)
\Gamma\left(\tfrac34+s\right)\\
\times
\left[
D_{2s-\frac12}((1-i)\xi)D_{-2s-\frac12}((1-i)\eta)
+D_{2s-\frac12}((1+i)\xi)D_{-2s-\frac12}((1+i)\eta)
\right]
ds.
\end{multline*}
\end{corollary}

\begin{corollary}\label{3:BesselY1}
Suppose $\xi,\eta\in\C$ satisfy $|\Im\eta|<\Re \xi$ and $|\Im \xi|<\Re \eta$.
Let $\sigma\in\C$ be such that $|\Re \sigma|<\frac34$.
Then,
\begin{multline*}
\hspace{0.5cm}
 Y_1\left(\tfrac12(\xi^2+\eta^2\right)
=\frac{2^{1/2}(\xi^2+\eta^2)}{\pi^2\xi\eta}
\int_{\sigma-i\infty}^{\sigma+i\infty}
\Gamma\left(\tfrac34-s\right)\Gamma\left(\tfrac34+s\right)\\
\times
\left[
D_{2s-\frac12}((1-i)\xi))D_{-2s-\frac12}((1-i)\eta)
-D_{2s-\frac12}((1+i)\xi)D_{-2s-\frac12}((1+i)\eta)
\right]
ds.
\end{multline*}
\end{corollary}

One can additionally verify \eqref{3:addJ0} through analysis found in \cite{CohlVolkmer}.
By using \cite[Theorem 3.2]{CohlVolkmer}, \cite[(14)]{CohlVolkmer} and
\cite[(40)]{CohlVolkmer}, with $(\xi_0,\eta_0)=(0,0)$ and noting that $u_2(t,0)=0,
u_1(t,0)=1.$ From this, one obtains an expansion for $J_0$, although it is
qualitatively different from the expansions derived in this paper. The expansion for $J_0$ derived
in \cite{CohlVolkmer} holds for all $(\xi,\eta)\in\mathbb{C}^2$, whereas the corresponding
formula \eqref{3:addJ0}) converges in the set
\[
T=\{(\xi,\eta)\in\C^2: |\Im \eta|<\Re \xi, \,|\Im \xi|<\Re \eta \}
\]
but it does not converge for all $(\xi,\eta)\in\C^2$. For example,
if $\xi=\eta=i$ then it follows from the asymptotics that the integral
in \eqref{3:addJ0} does not exist.
Note that \eqref{3:addJ0} is actually wrong when $\xi=\eta=0$.
This shows a fundamental difference between \eqref{3:addJ0} and its
corresponding expansion in \cite{CohlVolkmer}.

The Bessel function $J_0(z)$ is characterized as the solution of Bessel's
equation which is analytic on $\C$. If one gives a formula for
$J_0(\frac12(\xi^2+\eta^2))$ one would hope the formula to hold for all
$(\xi,\eta)\in\C^2$.
Thus, we believe that the previously derived expansion (found in \cite{CohlVolkmer}) is
superior.

Finally, using Corollaries \ref{2:cor1}, \ref{2:cor2} and noting \eqref{2:K}, we obtain integral addition theorems for $K_0$ and $K_1$.

\begin{corollary}\label{3:BesselK0}
Let $\xi,\eta\in\C$ satisfy
\begin{equation}\label{3:cond}
\Re((1-i)\xi+(1+i)\eta)>0\text{ and } \Re((1+i)\xi+(1-i)\eta)>0.
\end{equation}
Let $\sigma\in\C$ be such that $|\Re\sigma|<\frac14$.
Then,
\begin{equation}\label{3:addK0}
K_0\left(\tfrac14(\xi^2+\eta^2)\right)
=\frac{1}{2^{1/2}\pi i}
\int_{\sigma-i\infty}^{\sigma+i\infty}
\Gamma\left(\tfrac14-s\right)\Gamma\left(\tfrac14+s\right)
D_{2s-\frac12}(\xi)D_{-2s-\frac12}(\eta) ds.
\end{equation}
Remark: Condition \eqref{3:cond} implies that $\xi^2+\eta^2\in\C\setminus (-\infty,0]$. The Bessel function $K_0(z)$ on the left-hand side of \eqref{3:addK0} is evaluated in the sector $-\pi<\arg z<\pi$.
\end{corollary}

\begin{corollary}\label{3:BesselK1}
Let $\xi,\eta\in\C$ satisfy \eqref{3:cond}.
Let $\sigma\in\C$ be such that $|\Re\sigma|<\frac34$.
Then,
\[
\xi\eta K_1\left(\tfrac14(\xi^2+\eta^2)\right)
=\frac{2^{1/2}(\xi^2+\eta^2)}{\pi i}
\int_{\sigma-i\infty}^{\sigma+i\infty}
\Gamma\left(\tfrac34-s\right)\Gamma\left(\tfrac34+s\right)
D_{2s-\frac12}(\xi)
D_{-2s-\frac12}(\eta)
ds.
\]
\end{corollary}


\section*{Acknowledgments}

We express our gratitude to Adri Olde Daalhuis and Nico Temme
for their advice on asymptotics for the Kummer confluent hypergeometric
function of the second kind with complex arguments.




\begin{thebibliography}{10}

\bibitem{Buchholz}
H.~Buchholz.
\newblock {\em Die konfluente hypergeometrische {F}unktion mit besonderer
  {B}er\"ucksichtigung ihrer {A}nwendungen}.
\newblock Ergebnisse der angewandten Mathematik. Bd. 2. Springer-Verlag,
  Berlin, 1953.

\bibitem{CohlVolkmer}
H.~S. {Cohl} and H.~{Volkmer}.
\newblock {Eigenfunction expansions for a fundamental solution of Laplace's
  equation on $\R^3$ in parabolic and elliptic cylinder coordinates}.
\newblock {\em Journal of Physics A: Mathematical and Theoretical},
  45(35):355204, 2012.

\bibitem{NIST:DLMF}
{NIST Digital Library of Mathematical Functions}.
\newblock {Release 1.0.10 of 2015-08-07}.
\newblock {Online companion to \cite{NIST}}.

\bibitem{Erdelyi}
A.~Erd{\'e}lyi, W.~Magnus, F.~Oberhettinger, and F.~G. Tricomi.
\newblock {\em Higher Transcendental Functions. {V}ol. {I}}.
\newblock Robert E. Krieger Publishing Co. Inc., Melbourne, Fla., 1981.

\bibitem{NIST}
{F.~W.~J. Olver and D.~W. Lozier and R.~F. Boisvert and C.~W. Clark}, editor.
\newblock {\em {{NIST} {H}andbook of {M}athematical {F}unctions}}.
\newblock {Cambridge University Press}, {New York, NY}, {2010}.
\newblock {Print companion to \cite{NIST:DLMF}}.

\bibitem{Magnus41}
W.~Magnus.
\newblock Zur {T}heorie des zylindrisch-parabolischen {S}piegels.
\newblock {\em Zeitschrift f\"ur Physik. C. Particles and Fields},
  118:343--356, 1941.

\bibitem{Magnus}
W.~Magnus.
\newblock \"{U}ber eine {B}ezeihung zwischen {W}hittakerschen {F}unktionen.
\newblock {\em Nachrichten der Akademie der Wissenschaften in G\"ottingen. II.
  Mathematisch-Physikalische Klasse}, 1946:4--5, 1946.

\bibitem{Olver56}
F.~W.~J. Olver.
\newblock The asymptotic solution of linear differential equations of the
  second order in a domain containing one transition point.
\newblock {\em Philosophical Transactions of the Royal Society of London.
  Series A. Mathematical and Physical Sciences}, 249:65--97, 1956.

\bibitem{Temme}
N.~M. Temme.
\newblock Remarks on {S}later's asymptotic expansions of {K}ummer functions for
  large values of the $a$-parameter.
\newblock {\em Advances in Dynamical Systems and Applications}, 8(2):365--377,
  2013.

\bibitem{Volkmer2016}
H.~Volkmer.
\newblock {The asymptotic expansion of Kummer functions for large values of the
  $a$-parameter, and remarks on a paper by Olver}.
\newblock {\em {submitted to SIGMA}}, {2016}.

\end{thebibliography}

\def\cprime{$'$} \def\dbar{\leavevmode\hbox to 0pt{\hskip.2ex \accent"16\hss}d}

\end{document}